\documentclass[11pt]{amsart}
\usepackage{amsmath, amsthm, amssymb,url}
\usepackage{hyperref}

\author{J. Vandehey}
\thanks{Email: \href{mailto:vandehe2@illinois.edu}{\nolinkurl{vandehe2@illinois.edu}}}
\title[Irrationality of Lambert series]{On an incomplete argument of Erd\H{o}s on the irrationality of Lambert series}
\date{\today}

\newtheorem{thm}{Theorem}[section]

\newtheorem{lem}[thm]{Lemma}
\newtheorem{prop}[thm]{Proposition}

\begin{document}

\maketitle

\begin{abstract}
We show that the Lambert series $f(x)=\sum d(n) x^n$ is irrational at $x=1/b$ for negative integers $b < -1$ using an elementary proof that finishes an incomplete proof of Erd\H{o}s.  
\end{abstract}

\section{Introduction}

Chowla \cite{chowla} conjectured that the functions 
\[
f(x) = \sum_{n=1}^\infty \frac{x^n}{1-x^n} \qquad \text{and} \qquad g(x) = \sum_{n=1}^\infty \frac{x^n}{1-x^n}(-1)^{n+1}
\]
are irrational at all rational values of $x$ satisfying $|x|<1$.  For such $x$ the above functions may be rewritten as 
\[
f(x) = \sum_{n=1}^\infty d(n) x^n \qquad \text{and} \qquad g(x) = \frac{1}{4}\sum_{n=1}^\infty r(n)x^n,
\]
where $d(n)$ is the number of divisors of $n$ and $r(n)$ is the number of representations of $n$ as a sum of two squares.

Erd\H{o}s \cite{erdos1} proved that for any integer $b>1$, the value $f(1/b)$ is irrational.  He did so by showing that $f(1/b)$ written in base $b$ contains arbitrary long strings of $0$'s without terminating on $0$'s completely.  If we take $b<-1$ to be a negative integer, then Erd\H{o}s' methods show by the same method that $f(1/b)$ in base $|b|$ contains arbitrary long strings of $0$'s; however, Erd\H{o}s claims without proof that showing it will not terminate on $0$'s can be done using similar methods.  It is not clear what method Erd\H{o}s intended, and in later papers (including his review of similar irrationality results \cite{erdos2}) Erd\H{o}s only refers to proving the case of positive $b$.

Since then, several proofs have been offered for the irrationality of the $b<-1$ case and far more general theorems besides.  Much credit is often given to Bezivin \cite{bezivin} and Borwein \cite{borwein} for proving the first major generalizations of these results; and other results can be often be found in the literature under the term of the $q$-analogue of the logarithm or, simply, the $q$-logarithm.  However, these results are proved using entirely different techniques than what Erd\H{o}s uses and leaves open the question of whether his method could have finished the proof.

Erd\H{o}s' method can be extended to the following stronger result with a virtually identical proof.

\begin{thm}\label{thm:erdos}
Let $b>1$ be a positive integer and $\mathcal{A}$ be any finite set of non-negative integers.  Then for any sequence $\{a_n\}_{n=1}^\infty$ taking values in $\mathcal{A}$ such that the sequence does not end on repeated $0$'s, we have that
\[
\sum_{n=1}^\infty d(n) \frac{a_n }{b^n}
\]
is irrational.
\end{thm}

Theorem \ref{thm:erdos} has the following curious corollary.  Let $a_n(x)$ be the $n$th base $b$ digit of a number $x$ in $(0,1)$.  (If $x$ has two base $b$ expansions, then we chose the one which does not end on repeated $0$'s.)  Then the map
\[
x=\sum_{n=1}^\infty \frac{a_n(x)}{b^n} \longmapsto \sum_{n=1}^\infty  d(n) \frac{a_n(x)}{b^n}
\]
has its image in $\mathbb{R}\setminus \mathbb{Q}$ and is also continuous at all $x$ that do not have a representation as a finite base $b$ expansion.

We could replace the condition that $a_n$ be in the finite set $\mathcal{A}$ with a restriction that $0\le a_n \le \phi(n)$ for some sufficiently slowly growing integer-valued function $\phi$.  It would be interesting to know what the fastest growing $\phi$ for which the Theorem \ref{thm:erdos} holds would be.

In this paper, we will prove the following extension of Theorem \ref{thm:erdos}. 

\begin{thm}\label{thm:main}
Let $b>1$ be a positive integer and $\mathcal{A}$ be any finite set of integers \emph{that does not contain} $0$.  Then for any sequence of $\{a_n\}_{n=1}^\infty$ taking values in $\mathcal{A}$, we have that
\[
\sum_{n=1}^\infty d(n)\frac{a_n}{b^n}
\]
is irrational.
\end{thm}

The new ingredient to extend Erd\H{o}s' method is finding arbitrarily long strings of zeros \emph{that are known to be preceded by a non-zero number}, and to find these strings arbitrarily far into the base $|b|$ expansion.

In particular, by taking $a_n=(-1)^n$, this proves that $f(1/b)$ is irrational for negative integers $b<1$ as well, completing Erd\H{o}s' proof.

\section{Proof of Theorem \ref{thm:main}}

We will require a result mentioned by Alford, Granville, and Pomerance \cite[p.~705]{agp}.  The function $\pi(N;d,a)$ equals the number of primes up to $N$ that are congruent to $a$ modulo $d$.

\begin{prop}\label{prop:agb}
Let $0<\delta<5/12$.  Then there exist positive integers $N_0$ and $\overline{\mathcal{D}}$ dependent only on $\delta$, such that the bound
\[
\pi(N;d,a) \ge \frac{N}{2\varphi(d) \log N}
\]
holds for all $N>N_0$; all moduli $d$ with $1 \le d \le N^\delta$, except, possibly for those $d$ that are multiples of some element in $\mathcal{D}(N)$, a set of at most $\overline{\mathcal{D}}$ different integers that all exceed $\log N$; and all a relatively prime to $d$.
\end{prop}

We begin our proof much as Erd\H{o}s did his.  Let $b \ge 2$ be a fixed positive integer, let $\mathcal{A}$ be a finite set of integers that does not contain $0$, and let $N$ be a large positive integer that is allowed to vary.  Define $k$ in terms of $N$ by 
\[
k=k(N):=\lfloor \left( \log{N} \right)^{1/10} \rfloor.
\]
Let $j_0$ be a fixed integer, independent of $N$, so that $2 \max_{a\in \mathcal{A}} |a| / b^{j_0} <1$.

Let $0<\delta<5/12$ be some sufficiently small fixed constant, and let $N_0=N_0(\delta)$ and $\overline{\mathcal{D}}=\overline{\mathcal{D}}(\delta)$ be the corresponding constants from Proposition \ref{prop:agb}.  Let $N_1 >N_0$ be large enough so that for any $N> N_1$, the interval $((\log N)^2, 2 (\log N)^2)$ cotains at least $u+\overline{\mathcal{D}}$ primes, where $u=u(N)=k(k-1)/2$.  In addition, for such $N>N_1$, let $\mathcal{D}(N)$ be the set of exceptional moduli from Proposition \ref{prop:agb}.  Since we assume that $\delta$ is constant, $|\mathcal{D}(N)|\le \overline{\mathcal{D}}$ is bounded.  

For each $D$ in $\mathcal{D}(N)$, let $\tilde{p}_D$ denote the smallest prime strictly greater than $(\log N)^2$ that divides $D$, if such a prime exists, and then let $p_1< p_2< \dots <p_u$ be the smallest $u$ primes strictly greater than $(\log N)^2$ that are not equal to $\tilde{p}_D$ for any $D\in \mathcal{D}(N)$; by assumption on $N$, we have that each such $p_i$ is less than $2(\log N)^2$.  Finally, let
\[
A:= \prod_{i=1}^{ j_0(j_0-1)/2}p_i^b \prod_{i=j_0(j_0+1)/2+1 }^{ u} p_i^b,
\]
so that, in particular, $A$ is not a multiple of any $D$ in $\mathcal{D}(N)$; moreover, provided $N$ is sufficiently large, we have
\[
A<(2(\log N)^2)^{bk(k-1)/2} \le N^\delta.
\]

  By the Chinese remainder theorem, there exists an integer $r$, with $0 \le r \le A-1$, such that 
\[
r+j\equiv \prod_{i=j(j-1)/2+1}^{j(j+1)/2} p_i^{b-1} \pmod{\prod_{i=j(j-1)/2+1}^{j(j+1)/2} p_i^{b}}, \qquad 0 \le j \le k-1, j\neq j_0.
\]
(The exception $j\neq j_0$ marks the key difference between this proof and Erd\H{o}s'.)  Since all the $p_i$'s are bounded below by $(\log N)^2$, we have that $r$ necessarily tends to infinity as $N$ does, although possibly much slower.

With this value of $r$, for any integer of the form 
\[
r+mA, \qquad 0 \le m <\lfloor N/A \rfloor,
\]
we have that
\begin{equation}\label{eq:congruence}
d(r+mA+j) \equiv 0 \pmod{b^{j+1}}, \qquad 0 \le j <k,\ j \neq j_0
\end{equation}
by the multiplicity of $d(\cdot)$.  Moreover, $r+j_0$ is relatively prime to $A$, since each $p$ dividing $A$ also divides some $r+j$, with $0 \le j <k,$ $j \neq j_0$; the largest $j$ can be is $k\le (\log N)^{1/10}$, but all primes dividing $A$ are at least $(\log N)^2$.

We can also apply Proposition \ref{prop:agb} to see that
\begin{equation}\label{eq:agbcor}
\pi (N,A,r+j_0)\ge \frac{N}{2\varphi(A)\log N}.
\end{equation}

Erd\H{o}s in \cite{edros1} also proved the following result, which we give here without reproof.  (While our construction of $A$ and $r$ are different from Erd\H{o}s', they satisfy all the requirements for Erd\H{o}s' proof technique to still hold.)

\begin{lem}\label{lem:erdos}
With $A$, $r$, $b$, and $k$ all as above, the number of $m < \lfloor N/A \rfloor$ such that 
\[
\sum_{n> r+k+mA} d(n)\frac{1}{b^n} > \frac{1}{b^{r+k/2+mA}}
\]
is less than 
\[
\frac{10cN(\log{N})^2}{A2^{k/4}}
\]
for some constant $c$ independent of all variables.
\end{lem}

Regardless of how large $c$ is, we have, for sufficently large $N$, that 
\[
\frac{N}{2\varphi(A)\log N} \ge \frac{10cN(\log{N})^2}{A2^{k/4}}.
\]
Therefore, by combining Lemma \ref{lem:erdos} with \eqref{eq:congruence} and \eqref{eq:agbcor}, we see that for sufficiently large $N$ there exists some $m_0 < \lfloor N/A\rfloor$ such that
\begin{equation}\label{eq:upper}
b^{j+1} | d(r+m_0 A+j) , \qquad 0 \le j <k, j \neq j_0,
\end{equation}
\begin{equation}\label{eq:middle}
r+m_0 A + j_0 \text{ is prime},
\end{equation}
and
\begin{equation}\label{eq:lower}
\sum_{n> r+k+m_0 A} \left| d(n) \frac{a_n }{b^n} \right| \le \frac{\max_{a\in \mathcal{A}} |a|}{b^{r+k/2+mA}}.
\end{equation}

Now consider a particular sequence $(a_n)$ with each $a_n \in \mathcal{A}$ together with the sum
\[
\sum_{n=1}^\infty d(n) \frac{a_n }{b^n} .
\]
By \eqref{eq:upper}, the partial sum
\[
\sum_{\substack{n \le r+k+m_0 A \\  n \neq r+j_0+m_0 A}} d(n) \frac{a_n }{b^n},
\]
when written in base $b$, has its last non-zero digit in the $(r-1+m_0A)$th place or earlier.\footnote{Here we switch back to the convention that finite expansions are assumed to end on an infinite string of zeros.}  In addition, by \eqref{eq:lower}, the partial sum
\[
\sum_{n > r+k+m_0 A} d(n) \frac{a_n }{b^n}
\]
when written in base $b$ has its first non-zero digit in the \[(r+k/2+m_0 A-\lceil \log_b \max_{a\in \mathcal{A}} |a| \rceil)\text{th}\] place or later.  The number
\[
d(r+j_0+m_0A) \frac{a_{r+j_0+m_0A} }{b^{r+j_0+m_0A}}=\frac{2a_{r+j_0+m_0A} }{b^{r+j_0+m_0A}}
\]
when written in base $b$ has its non-zero digits only between the $(r+m_0  A)$th and $(r+j_0+m_0 A)$th place, and it has at least one such non-zero digit.  Thus the full sum has a string of at least $k/2+O(1)$ zeroes immediately preceded by a non-zero digit starting somewhere between the $(r+m_0 A)$th and $(r+j_0+m_0 A)$th place.

So as $N$ increases to infinity, we can find arbitrarily long strings of $0$'s (which corresponds to $k$ increasing to infinity) immediately preceded by a non-zero digit, and we find these strings arbitrarily far out in the expansion (since $r$ also tends to infinity).  The base $b$ digits cannot therefore be periodic and hence the sum is irrational.  This completes the proof.

\section{Acknowledgements}

The author acknowledges support from National Science Foundation grant DMS 08-38434 ``EMSW21-MCTP: Research Experience for Graduate Students.''  The author would also like to thank Paul Pollack and Paul Spiegelhalter for their assistance.

\end{document}